\documentclass[a4paper,reqno, 11pt]{amsart}

\usepackage[utf8]{inputenc}
\usepackage{tikz}
\usepackage{tikz-cd}
\usepackage[top=2.7cm,bottom=2.7cm,left=2.3cm,right=2.3cm]{geometry}
\usepackage{graphicx}

\usepackage{hyperref} 
\usepackage[dvips]{epsfig}
\usepackage{color}
\usepackage{verbatim}
\usepackage{amsgen, amstext,amsbsy,amsopn,amssymb, amsthm, 
mathptmx, amsfonts,amssymb,amscd,amsmath,nicefrac,euscript,enumerate,url,verbatim,calc,framed}
\usepackage[all]{xy}
\def\red#1{{\color{red}#1}}

\theoremstyle{plain}
\newtheorem{thm}{\bf Theorem}[section]

\newtheorem{prop}[thm]{\bf Proposition}
\newtheorem{lem}[thm]{\bf Lemma}
\newtheorem{cor}[thm]{\bf Corollary}

\theoremstyle{definition}
\newtheorem{defn}[thm]{\bf Definition}

\theoremstyle{remark}
\newtheorem{rem}[thm]{\bf Remark}
\newtheorem{quest}[thm]{\bf Question}

\newtheorem{note}[thm]{\bf Note}

\DeclareMathOperator{\reg}{reg}
\DeclareMathOperator{\Tor}{Tor}
\DeclareMathOperator{\lcm}{lcm}

\def\NN{\mathbb{N}}
\def\ZZ{\mathbb{Z}}

\newsavebox\foobox

\begin{document}

\title{\textbf{Rees algebra of Maximal Order Pfaffians and its diagonal subalgebras}}

\author[Neeraj Kumar]{Neeraj Kumar}
\address{Department of Mathematics, IIT Hyderabad, Kandi, Sangareddy - 502285}
\email{neeraj@math.iith.ac.in}

\author[Chitra Venugopal]{Chitra Venugopal}
\address{Department of Mathematics, IIT Hyderabad, Kandi, Sangareddy - 502285}
\email{ma19resch11002@iith.ac.in}


\date{\today}

\begin{abstract}
Given a skew-symmetric matrix $X$, the Pfaffian of $X$ is defined as the square root of the determinant of $X$. In this article, we give the explicit defining equations of the Rees algebra of a Pfaffian ideal $I$ generated by the maximal order Pfaffians of a generic skew-symmetric matrix. 
We further prove that all diagonal subalgebras of the corresponding Rees algebra of $I$ are Koszul. We also look at Rees algebras of Pfaffian ideals of linear type associated with certain sparse skew-symmetric matrices. In particular, we consider the tridiagonal matrices and identify the corresponding Pfaffian ideals to be of Gr\"obner linear type and as the vertex cover ideals of unmixed bipartite graphs. As an application of our results, we conclude that all their ordinary and symbolic powers have linear quotients. 
\end{abstract}

\maketitle

\section*{Introduction}
Let $A$ be a graded Noetherian ring and $I$ be a homogeneous ideal of $A$. Then the Rees algebra of $I$ denoted by $\mathcal{R}(I)$ is a bigraded algebra defined as $\oplus_{i \geq 0}I^{i}$. The Rees algebra of an ideal $I$ forms an important class of bigraded algebra, which contains a great deal of information about the powers of the
ideal $I$. Moreover, geometrically it corresponds to the blowup
of Spec$(A)$ along the variety of $I$.
In general, corresponding to a homogeneous ideal $I$ of a ring $A$, finding the explicit defining equations of the Rees algebra is not easy. Some study has been done for certain classes of ideals like perfect ideals of grade $2$ \cite{Morey, MoreyUlrich, Rees_linear_presented,HerSimVas}, perfect Gorenstein ideals of grade $3$ \cite{KPU,Morey}, determinantal ideals \cite{BCV1,BCV2} etc. 
An ideal $I$ of a ring $A$ is said to be of \textit{linear type} if the Rees algebra of $I$ is isomorphic to its symmetric algebra. We are interested in the study of Rees algebras of ideals generated by $d$-sequences (a notion introduced by Huneke in \cite{CH,CH1}) as they form a class of ideals of linear type. The motivation to explore the Rees algebra corresponding to a $d$-sequence comes from the analogous study in \cite{CHTV,SimisTrungValla,CC, Neeraj} for an ideal generated by a regular sequence. 

\vspace{2mm}

In this article, we look at a particular class of ideals called the Pfaffian ideals, which come corresponding to skew-symmetric matrices. Let $X$ be a skew-symmetric matrix, and let det $X$ denote its determinant.
Then \textit{Pfaffian of $X$} denoted by $\text{Pf}(X)$ is defined as the square root of det $X$ $i.e.,$ $\text{Pf}(X)^2=\text{det } X$ (cf.  \cite{Artin}). The \textit{Pfaffian ideal of X} denoted by $\text{Pf}_{n-1}(X)$ is the ideal obtained by considering the Pfaffians of submatrices of order $n-1$ obtained by deleting a row and the corresponding column of the matrix $X$ (cf. \cite{DBDE}). In \cite{DBDE}, Buchsbaum and Eisenbud proved that every Gorenstein ideal of codimension $3$ in a commutative Noetherian ring $A$ can be identified as the ideal of Pfaffians of order $(n-1)$ of some $n \times n$ alternating matrix of rank $n-1$.
Under some assumptions on the entries of the skew-symmetric matrices, Pfaffian ideals are found to be of linear type (cf. \cite{Baetica}, \cite{Price&Cooper}). 
We attempt to study the diagonal subalgebras of the corresponding Rees algebras of such Pfaffian ideals.

\vspace{2mm}

The notion of diagonal subalgebras was introduced by Simis, Trung and Valla in \cite{SimisTrungValla}, generalizing the concept of the Segre product of graded algebras. The diagonal subalgebras of certain classes of equigenerated homogeneous ideals of a standard graded polynomial ring can be viewed as the homogeneous coordinate rings of some rational varieties embedded in projective spaces (cf. \cite{SimisTrungValla}). It is also known that 
for $c,e \in \NN$ and a homogeneous ideal $I$ of $B=K[x_1, \ldots,x_n]$, if $I^e$ is generated by forms of degree $d \leq c-1$, then $K[(I^e)_c]$ can be identified as a diagonal subalgebra of the Rees algebra in a natural way \cite{SimisTrungValla,CHTV}.
One of the key challenges in the study of diagonal subalgebras is to find suitable conditions on a bigraded algebra $R$ such that certain algebraic properties of $R$ are inherited by $R_{\Delta}$ (Definition \ref{diagonal_sub}).

\vspace{2mm}

We are interested in the Koszulness and Cohen-Macaulay property of the diagonals of Rees algebras of equigenerated homogeneous ideals. There is much literature on these properties of diagonals of bigraded algebras (cf. \cite{SimisTrungValla,CHTV,Olga,Neeraj,AKM}). \vspace{2mm}

A standard graded $K$-algebra $A$ is {\emph{Koszul}} if the non-zero entries of the matrices representing the maps in the minimal free resolution of $K$ are homogeneous of degree $1$. Several articles have discussed the Koszul property of the diagonal subalgebras of Rees algebras of ideals (cf. \cite{CHTV, CC,Fr,SB,Neeraj,AKM}.  Explicit lower bounds are known for the residual intersections (\cite{AKM}) and when the ideals are complete intersections \cite{CHTV,Neeraj}. More generally, in \cite{CHTV}, it has been proved that for any standard bigraded $K$-algebra $R$, $R_{\Delta}$ is Koszul for $c,e \gg 0$.  

\vspace{2mm}

In \cite{SimisTrungValla} for $\Delta=(1,1)$, the authors discuss some classes of ideals for which the $\Delta$-diagonal of the corresponding Rees algebras are Cohen-Macaulay. Complete intersections and certain classes of straightening closed ideals in algebras with straightening law (like the determinantal ideals generated by the maximal minors of generic matrices) are some of the ideals looked at.
In \cite{CHTV}, Conca, Herzog Trung and Valla solve an open problem posed in \cite{SimisTrungValla} regarding the conditions on $(c,e)$, which guarantees Cohen-Macaulay property of $K[(I^e)_c]$  when $I \subset B$ is a homogeneous complete intersection minimally generated by $r$ forms of degree $d_1, \ldots, d_r$. 
For some classes of perfect ideals of height two as well, certain bounds on $c,e$ are known, for which the diagonals of the corresponding Rees algebras are Cohen-Macaulay (cf. \cite{AKM}).
In general, it is known that if a standard bigraded ring $R$ is Cohen-Macaulay, then $R_{\Delta}$ is Cohen-Macaulay for large integers $c \gg e>0$ (cf. \cite{Olga}). 

\vspace{2mm}

In this article, we primarily look at the equigenerated Pfaffian ideals so that the associated Rees algebras are standard
bigraded, thus forming standard graded K-algebras with respect to the total degree. In particular, it makes sense to look at the Koszul property of such graded K-algebras. 

\vspace{2mm}

Some of the important results discussed in this article are the following. 
\begin{enumerate}[(i)]
    \item Let $I=$Pf$_{n-1}(X)$ where $X$ is a generic skew-symmetric matrix of odd order $n=2r+1$, $r \in \NN$ (defined in Section \ref{Preliminaries}). Then we prove the following.
    \begin{enumerate}[(a)]
        \item  The $c$-th Veronese subalgebra of the corresponding Pfaffian ring is Koszul for $c \geq n/4.$
        \item  $\mathcal{R}(I)$ is generated by quadrics but need not be Koszul always.
        \item  All the diagonals of $\mathcal{R}(I)$ are Koszul.
    \end{enumerate}\vspace{2mm}
\item For the Pfaffian ideal $I$ coming from a tridiagonal matrix of the form mentioned in Theorem \ref{tridiagonal}, $\mathcal{R}(I)_{\Delta}$ is Koszul and Cohen-Macaulay for all $\Delta$. In this case, Pf$_{n-1}(X)$ is an equigenerated monomial ideal of Gr\"obner linear type (defined in Section \ref{sparse_skew}), which can be identified as the vertex cover ideal of an unmixed bipartite graph, thereby giving information about its ordinary and symbolic powers having linear quotients. \vspace{2mm}
\item For $I=$Pf$_{n-1}(X)$ where $X$ has the form given in Proposition \ref{pf_case2}, $\mathcal{R}(I)_{\Delta}$ is Koszul and Cohen-Macaulay for all $\Delta$.  

\end{enumerate}

\vspace{1mm}

We mainly focus on the maximal order Pfaffians, as in the non-maximal case, the defining ideals of Rees algebras are not necessarily generated by quadrics, a property essential for studying the Koszulness of Rees algebras.

\vspace{2mm}

The reader may be familiar with some of the observations made in this article. However, to maintain the study self-contained, we reproduce some arguments and independently establish the results. All the computations in this article are done using Macaulay2 (\cite{M2}). 

\vspace{2mm}

\section{Preliminaries} \label{Preliminaries}

\vspace{2mm}

We consider $K$ to be a field of characteristic zero throughout the article. Consider the skew-symmetric matrix $X= \begin{bmatrix}
                                        0      & x_{12} & \ldots & x_{1\;n}\\
                                        -x_{12} &   0    & \ldots & x_{2\;n}\\
                                        \vdots & \vdots & \ddots & \vdots\\
                                        -x_{1\;n} & -x_{2\;n} & \ldots &   0   \\
                                       \end{bmatrix}$
of odd order $n=2r+1$, $r \in \NN \cup \{ 0\}$, where the entries $x_{ij}$ for $i=1, \ldots, n-1$, $j=2, \ldots,n$
are indeterminates. This is the form of a generic skew-symmetric matrix of odd order. Let $B=K[X]$ where $K[X]$ is the polynomial ring in indeterminates being the non-zero entries of $X$ and $I=\text{Pf}_{n-1}(X)=\text{Pf}_{2r}(X).$ Then $B/I$ is said to be a \textit{Pfaffian ring}. 

\begin{note}
Let $X$ be a skew-symmetric matrix of odd order. Then by Pf$_{\bar{\ell}}(X)$ we mean the Pfaffian of the submatrix of $X$ obtained by removing its $\ell\text{-th}$ row and the $\ell\text{-th}$ column.
\end{note}

\vspace{1mm}

For a field $K$, $A$ is said to be a standard graded $K$-algebra if $A=\oplus_{i \geq 0} A_i$ such that $A_0=K$,  $A_1$ is a finite dimensional $K$-vector space and $A_iA_j=A_{i+j}$ for every $i,j\in \NN \cup \{0\}$. 
Any standard graded $K$-algebra $A$ can be identified as $B/I$, where $B$ is a standard graded polynomial ring over $K$ ($K$-algebra) and $I$ is its homogeneous ideal. For an equigenerated ideal $\mathfrak{I}$ of $A$, the \textit{Rees algebra of $\mathfrak{I}$ in $A$} is a standard bigraded $K$-algebra defined as $\mathcal{R}(\mathfrak{I})= \bigoplus_{n\geq0} {\mathfrak{I}}^n$, where standard bigraded means $\mathcal{R}(\mathfrak{I})_1=\oplus_{i\geq 0} \mathcal{R}(\mathfrak{I})_{(i,0)}$ and $\mathcal{R}(\mathfrak{I})_2=\oplus_{i\geq 0} \mathcal{R}(\mathfrak{I})_{(0,i)}$ are standard graded $K$-algebras.
Similar to the graded case, $\mathcal{R}(\mathfrak{I})$ has a presentation of the form $S/J $ as a quotient of the standard bigraded polynomial ring $S=K[X,Y]$ over a field $K$ with the degree of the variables in $X$ being $(1,0)$ and that in $Y$ being $(0,1)$, by a bihomogeneous ideal $J$. An ideal $\mathfrak{I}$ of a ring $A$ is of linear type if the defining relations of  $\mathcal{R}(\mathfrak{I})$ are linear in the indeterminates $Y$.

\vspace{2mm}

\begin{defn} \label{diagonal_sub}
For two integers $c,e \geq 0$ with $(c,e) \neq (0,0)$, the $(c,e)$-diagonal is $\Delta= \{(cs,es):s \in \ZZ \}$ of $\ZZ^2$. The diagonal subalgebra of a bigraded algebra $R$ along $\Delta$ is defined as the graded algebra $R_{\Delta}:=\bigoplus_{i \geq 0} R_{(ci,ei)}$ (cf. \cite{SimisTrungValla}). 
\end{defn}

It is analogous to the notion of Veronese subalgebras of graded algebras. For $c \in \NN$ and a standard graded $K$-algebra $A$, $A^{(c)}=\oplus_{j \in \NN}A_{cj}$ is defined as the $c$-th Veronese subalgebra of $A$. 

\vspace{2mm}

In \cite{SimisTrungValla}, the authors have given the presentation of a diagonal subalgebra of a bigraded algebra $R$ in the following way. 
Consider $R \cong S/J$, where $S=K[x_1,\ldots,x_n,y_1, \ldots ,y_m]$ is a bigraded polynomial ring and $J$ a bihomogeneous ideal of $S$. Then 
\begin{equation} \label{presentation}
R_{\Delta} =S_{\Delta}/J_{\Delta}, 
\end{equation}
where $S_{\Delta}$ is the Segre product of $K[x_1, \ldots, x_n]^{(c)}$ with $K[y_1, \ldots, y_m]^{(e)}$ and  $J_{\Delta}=\oplus_{i \geq 0}J_{(ci,ei)}$.

\vspace{1mm}

Now let $\Delta=(1,1)$. Then $S_{\Delta}= K[x_iy_j| \, 1 \leq i \leq n, \, 1 \leq j \leq m]$ and a presentation of $S_{\Delta}$ can be seen as $S_{\Delta} \cong K[T]/I_2(T)$ where $T=(t_{ij})$ is an $n \times m$ matrix of indeterminates and $I_2(T)$ is the ideal generated by the $2$-minors of $T$. This presentation is obtained by mapping $t_{ij}$ in $K[T]$ to $x_iy_j$ in $S_{\Delta}$. The following lemma gives the form for the generators of $J_{\Delta}$ in $K[T]$ when $\Delta=(1,1)$.

\vspace{1mm}

\begin{lem} (\cite[Lemma 2.1]{SimisTrungValla}) \label{present2}
Let $R \cong S/J$ be a standard bigraded $K$-algebra and $J$ a bihomogeneous ideal of $S$ generated by $g_1, \ldots, g_r$ with deg $g_i=(a_i,b_i)$. Then for $c_i=\text{max} \{a_i,b_i\}$, the generators of $J_{\Delta}$ has the form $g_im$ where $m$ is a monomial of degree $(c_i-a_i,c_i-b_i)$, $i=1, \ldots, r$.     
\end{lem}

For an equigenerated ideal of a standard graded polynomial ring, if $c \geq ed+1$, then the dimension of the corresponding diagonal subalgebra of the Rees algebra is found to be independent of the diagonal.

\vspace{1mm}

\begin{lem} (\cite[Lemma 1.3(ii)]{CHTV}) \label{dimension}
Let $I$ be an equigenerated ideal of a standard graded polynomial ring in $n$ indeterminates with the degree of the generators denoted by $d$. If $c \geq ed+1$, then the $\dim \, \mathcal{R}(I)_{\Delta}=n$ where $\Delta=(c,e)$.
\end{lem}

\vspace{1mm}

Let $A$ be a standard graded $K$-algebra and $M$ be a finitely generated $A$-module. Let $t_i^A(M)= \sup \{j: \Tor_i^A(M,K)_{j} \neq 0  \}$ with $t_{i}^A(M)=- \infty$ if $\Tor_i^A(M,K)_{j}=0$ for all $j \geq 0$. Then the regularity of $M$ denoted by reg$_A(M)$ is defined as $\reg_A(M)=\sup \{ t_{i}^A(M) -i, \, i \geq 0\}$. Similarly, corresponding to a standard bigraded $K$-algebra, there is an analogous notion of $x$-regularity and $y$-regularity (refer \cite[Section 2]{ACN} for definitions).

\vspace{2mm}

Let $M$ be an $A$-module generated by elements of the same degree, say $d$. Then $M$ is said to have a \textit{$d$-linear resolution} over $A$ if reg$_A(M)=d$. If the degree $d$ of the generators of an $A$-module $M$ is clear from the context, then just the terminology 'linear resolution' is used. A standard graded $K$-algebra $A$ is \textit{Koszul} if the minimal $A$-free resolution of the residue field is linear, that is, reg$_A(K)=0$. 

\vspace{2mm}

\begin{rem} \label{Koszulness_linear}
We recall the following results related to the Koszulness of bigraded algebras and the linear resolutions of ideals for later use.  

\begin{itemize}
    \item[(i)] Let $R$ be a standard bigraded $K$-algebra with the free modules in its minimal bigraded $S$-free resolution being denoted by $F_i=\oplus_{(a,b) \in \NN^2}S(-a,-b)^{\beta_{i,a,b}}$. Then for $\Delta=(c,e)$, $R_{\Delta}$ is Koszul if max$\{ \dfrac{a}{c},\dfrac{b}{e}: \beta_{i,a,b} \neq 0 \} \leq i+1$ for every $i$ (\cite[Theorem 6.2]{CHTV}). This shows how information about the shifts in the $S$-free resolution of $\mathcal{R}(I)$ helps in obtaining the lower bounds for $(c,e)$ for which $\mathcal{R}(I)_{\Delta}$ is  Koszul.
    \item[(ii)] If $R$ is a Koszul bigraded $K$-algebra then $R_{\Delta}$ is Koszul for all $\Delta$ (\cite[Theorem 2.1]{SB}).
    \item[(iii)] Let $I$ be an equigenerated ideal of a standard graded ring $A$. If $\mathcal{R}(I)$ is Koszul then $I^n$ has a linear resolution for all $n \geq 0$ (\cite[Corollary 3.6]{SB}).
    
\end{itemize}
\end{rem}

\vspace{3mm}
Similarly, there are results which help in getting bounds for $\Delta$ from the $S$-free resolution of $\mathcal{R}(I)$ such that $\mathcal{R}(I)_{\Delta}$ is Cohen-Macaulay.

\begin{lem} \cite[Lemma 3.10]{CHTV} \label{CM_criteria}
Let $S=K[x_1, \ldots x_n, y_1, \ldots, y_m]$ be a standard bigraded polynomial ring. Then for $\Delta=(c,e)$ and $a,b \in \ZZ$, 

\begin{itemize}
    \item[(i)] $\dim$ $ S(-a,-b)_{\Delta}=n+m-1$.
    \item[(ii)] If $0<b<m$ or $0<a<n$, then $S(-a,-b)_{\Delta}$ is Cohen-Macaulay if $c> \max \{-a,-n+a \}$ and $e> \max \{ -b,-m+b \}$. 
\end{itemize}
\end{lem}

\vspace{2mm}

\vspace{2mm}

The following result in \cite{D.Taylor} gives the defining relations (not necessarily minimal) of the Rees algebra of a monomial ideal.

\vspace{1mm}

Let $B=K[x_1, \ldots, x_n]$ and $I= \langle u_1, \ldots, u_s \rangle$
be a monomial ideal in $B$. Let $I_r$ denote the set of all sequences $\alpha=(i_1, \ldots, i_r)$ in $[s]$ of length $r$ such that $i_1 \leq i_2 \leq \ldots  \leq i_r$. For any $\alpha \in I_r$, let $u_{\alpha}=u_{i_1}u_{i_2}\ldots u_{i_r}$ and $t_{\alpha}=t_{i_1}t_{i_2} \ldots t_{i_r}$ and for any $\alpha, \beta \in I_r$, define $t_{\alpha,\beta}=\dfrac{ \lcm [u_{\alpha},u_{\beta}]}{u_{\beta}}t_{\beta}-\dfrac{\lcm[u_{\alpha},u_{\beta}]}{u_{\alpha}}t_{\alpha}$.
Then,
\begin{equation} \label{D.Taylor} 
\mathcal{R}(I) \cong B[t_1,\ldots,t_s]/J
\end{equation}
where $J=SJ_1+S( \cup_{r=2}^{\infty}J_r)$ with $J_r= \{ t_{\alpha,\beta}: \alpha,\beta \in I_r \}$.

\vspace{3mm}

\begin{defn}
Let $A$ be a commutative ring and $\{ \textbf{a} \}=\{ a_1, \ldots, a_n \}$ a sequence of elements in $A$. Then for an $A$-module $M$, 
$\{ \textbf{a} \}$ forms a \textit{$d$-sequence} on $M$ if the following holds:
\begin{enumerate}
    \item $a_iM \notin \langle a_1, \ldots, \hat{a_i}, \ldots, a_n \rangle M$ for $i=1, \ldots, n$.
    \item $(\langle a_1, \ldots, a_i \rangle M : a_{i+1}a_kM)= (\langle a_1, \ldots, a_i \rangle M :a_kM)$ for all $k \geq i+1$ and $i=0, \ldots, n-1$.
\end{enumerate}

\end{defn}

\vspace{1mm}

\section{Pfaffian ideals of generic skew-symmetric matrices} \label{generic_skew}

Assume $X$ to be a generic skew-symmetric matrix of odd order $n=2r+1$, $r \in \NN $ and $I=$Pf$_{n-1}(X)$. Then the ideal Pf$_{n-1}(X)$  generated by the maximal order Pfaffians is found to be of linear type (cf. \cite{Baetica}). Huneke proved that the Pfaffians in the generic case form a weak $d$-sequence \cite[1.20]{CH2} and further remarked that they seem to form a $d$-sequence \cite{CH1}. Recently, we gave a proof to show that they indeed form a $d$-sequence \cite{NK_CV}.
\footnote{The result is part of a preprint. The authors are willing to provide a copy of the preprint to the referee if needed.} 
In fact, from the proof, it is not difficult to see that they form an unconditioned $d$-sequence ($i.e.,$ $d$-sequence in any order).

\vspace{1mm}

The structure theorem of ideals of codimension $3$ \cite{DBDE} gives the minimal free resolution of a Pfaffian ring in the generic case corresponding to an alternating map. In the following lemma, we mention the differentials in the resolution (explicitly), which helps in studying the related Rees algebra.

\vspace{1mm}

\begin{lem} \cite[Theorem 2.1]{DBDE} \label{min_res_pfaffians}
Let $B=K[X]$ and $I=$Pf$_{n-1}(X)$. Then the minimal graded free resolution of $B/I$ has the following form.

\begin{equation} \label{min_free_res}
0 \longrightarrow B (-2r-1)\stackrel {d_3} \longrightarrow B^n(-r-1) \stackrel{d_2}\longrightarrow B^n(-r) \stackrel{d_1} \longrightarrow B \longrightarrow 0.    
\end{equation}

where $$ d_1= \begin{bmatrix}
              Pf_{\bar{1}}(X) & Pf_{\bar{2}}(X) & \cdots & Pf_{\bar{n}}(X) 
              \end{bmatrix},
      $$

    $$d_2= (a_{ij})= \left \{ \begin{array}{ccc}
     (-1)^{i+j}x_{n+1-i \; \; n+1-j}           & if & i>j \\
            0                   & if & i=j \\
     (-1)^{i+j+1}x_{n+1-j \; \; n+1-i}  & if & i<j. \\
       \end{array} \right. $$
       
and  $$ d_3= \begin{bmatrix}
              Pf_{\bar{1}}(X) \\
              Pf_{\bar{2}}(X) \\ 
              \vdots  \\
              Pf_{\bar{n}}(X) \\
              \end{bmatrix}  
     $$

\end{lem}

\vspace{1mm}

\begin{rem} \label{generic_rees}
As a consequence of the above result, for a generic skew-symmetric matrix $X$ and $I=$Pf$_{n-1}(X) \subseteq B=K[X]$, we observe the following.
\begin{enumerate}[(i)]
    \item reg $(B/I)=2r-2$.
    \item For a standard graded $K$-algebra $A$, it is known that $A^{(c)}$ is Koszul for $c \gg 0$ (cf. \cite{Backelin}). Using Lemma \ref{min_res_pfaffians} and Remark \ref{Koszulness_linear}(i)
    we have, $(B/I)^{(c)}$ is Koszul for $c \geq n/4.$ 
 
    \item Since $I$ is of linear type, the explicit defining relations of $\mathcal{R}(I)$ in $S=K[X,Y]$ will have the form 
$  d_2 . \begin{bmatrix}
   y_1 & y_2 & \ldots & y_n                     
\end{bmatrix}^T
$
where $d_2=(a_{ij})= \left \{ \begin{array}{ccc}
     (-1)^{i+j}x_{n+1-i \; \; n+1-j}           & if & i>j \\
            0                   & if & i=j \\
     (-1)^{i+j+1}x_{n+1-j \; \; n+1-i}  & if & i<j. \\
       \end{array} \right.$
 
\end{enumerate}
\end{rem}

\vspace{2mm}

Some interesting colon conditions are satisfied by the defining relations of the Rees algebra of the Pfaffian ideal in the generic case, which is discussed in the following Lemma. 

\vspace{1mm}

\begin{lem}\label{technical_lemma}
Consider the setup in Remark $\ref{generic_rees}(iii)$. Let the defining relations of $\mathcal{R}(I)=S/J$ obtained in the remark be denoted by $ g_1, \ldots, g_n $ where $S=K[X,Y]$ and $J= \langle g_1, \ldots, g_n \rangle$, a bihomogeneous ideal of $S$. Then the following holds.

\begin{enumerate}
    \item $ g_1, \ldots, g_{n-1} $ forms an $S$-regular  sequence.
    \item $(\langle y_n, g_1, \ldots, g_{n-1} \rangle:_S \text{Pf}_{\bar{1}}(X))= \langle y_1, \ldots, y_n \rangle.$
    \item $(\langle g_1, \ldots, g_{n-1} \rangle :_S g_n)= \langle g_1, \ldots, g_{n-1}, y_n, \text{Pf}_{\bar{1}}(X) \rangle$.
    \item $(\langle g_1, \ldots, g_{n-1} \rangle :_S y_n)=J.$
\end{enumerate}

\end{lem}

\begin{proof}   
Let $J'= \langle g_1, \ldots, g_{n-1} \rangle$ and $B=K[X]$.

\begin{enumerate}
    \item Consider the graded lexicographic term order on $S$ induced by $x_{12} > x_{23} > x_{34} > x_{45} > \cdots > x_{n-1 n}$ followed by the remaining indeterminates. Then from \cite[Lemma 2.2]{reg-seq}, we get that $g_1, \ldots, g_{n-1}$ forms an $S$-regular sequence. In fact, any $n-1$ relations among $g_1, \ldots, g_n$ can be proved to form a regular sequence by a corresponding change in the ordering of the indeterminates.
        \item For $1 \leq i,j \leq n$, the coefficient of $y_j$ in $g_i$ is given by the $(i,j)\text{-th}$ entry of the matrix 
            $$d_2=\begin{bmatrix}
                        0      & -x_{n-1\;n} & \ldots & -x_{2\;n} & x_{1 \;n}\\
                        x_{n-1 \;n} &   0    & \ldots & x_{2\;n-1} & -x_{1\;n-1}\\
                        \vdots & \vdots & \ddots & \vdots & \vdots\\
                        x_{2 \;n} & -x_{2 \; n-1} & \ldots & 0 & -x_{12} \\
                        -x_{1\;n} & x_{1\;n-1} & \ldots & x_{12} &   0   \\
                    \end{bmatrix}.$$ 
                    
\noindent Since the generic matrix has the form $X= \begin{bmatrix}
                        0      & x_{12} & \ldots & x_{1 \; n-1} & x_{1\;n}\\
                        -x_{12} &   0    & \ldots & x_{2\;n-1} & x_{2\;n}\\
                        \vdots & \vdots & \ddots & \vdots & \vdots \\
                        -x_{1\;n-1} & -x_{2\;n-1} & \ldots &   0  & x_{n-1 \; n} \\
                        -x_{1 \;n} & -x_{2 \; n} & \ldots & -x_{n-1 \; n} & 0 \\
                        \end{bmatrix}$, for each $i=1, \ldots,n-1$, we can represent  Pf$_{\bar{1}}(X)=  \sum_{\substack{j=1 \\ j \neq i }}^{n-1} \alpha_{ij}a_{ij}$ where $a_{ij}$ is the $(i,j)$-th entry of the matrix $d_2$ and $\alpha_{ij} \in B$ comes with respect to $a_{ij}$. 
                        From the form of the matrix $d_2$ and the representation of Pf$_{\bar{1}}(X)$ mentioned before, for each $i=1, \ldots, n-1$, one obtains 
                        the relations, 
                        \begin{equation} \label{pfaffians_rel}
                        \text{Pf}_{\bar{1}}(X)y_i= \sum _{\substack{j=1 \\ j \neq i }}^{n-1}\alpha_{ij}g_j\text{mod}(y_n).    
                        \end{equation}
                         Thus the inclusion $\supseteq$ follows. The equality can now be seen as a consequence of the primality of $\langle y_1, \ldots, y_{n
} \rangle$ and the inclusion $ \langle y_n, g_1, \ldots, g_{n-1} \rangle \subseteq \langle y_1, \ldots, y_n \rangle.$  
    \item Let $J''= \langle g_1, \ldots, g_{n-1}, y_n, \text{Pf}_{\bar{1}}(X) \rangle$. From the general form of $g_i$, $g_n \notin \langle g_1, \ldots, g_{n-1} \rangle$. Moreover, since the bidegree of $g_n$ is $(1,1)$ and the bidegrees of the generators of $J''$ are $(1,1)$, $(0,1)$ and $(r,0)$, $g_n \notin J''$. 
    
    It is not difficult to see that $g_n y_n=- (\sum_{i=1}^{n-1}g_i y_{i})$ and $g_n \text{Pf}_{\bar{1}}(X)= -(\sum_{i=1}^{n-1} g_{i}
    \text{Pf}_{\overline{n-i+1}}(X))$. Hence the inclusion $\supseteq$ follows. 
    Since $ J' \subseteq J''$, to prove the equality, it suffices to show that $J''$ is a prime ideal. This is equivalent to proving $\overline{\text{Pf}_{\bar{1}}(X)}$ is a prime element in $S/L$ where $L=\langle g_1, \ldots,g_{n-1}, y_n\rangle$ and $^{-}$ denotes an element in the respective quotient ring. First, we claim that Pf$_{\bar{1}}(X)$ is irreducible in $B$. To prove this, consider $f=\text{Pf}_{\bar{1}}(X)$ as a polynomial in one variable $x_{23}$ that is,  $f \in B'[x_{23}]$ where $ B'=K[X']$, $X'=X \backslash \{ x_{23}\}$. Let $B''$ be the field of fractions of $B'$. Then we can write $f=g(x')x_{23}+h(x')$ where $g(x'), h(x') \in B' \subseteq B''$ and $\text{gcd}(g(x'),h(x'))=1$. Then, clearly $f$ is an irreducible polynomial in $B''[x_{23}]$. Now assume that $f=p(x)q(x)$ where $p(x), q(x) \in B$. Then since $f$ is irreducible in $B''[x_{23}]$, without loss of generality let $p(x)$ be a unit in $B''[x_{23}]$. This implies $p(x) \in B''$, but since $p(x) \in B$, we have $p(x) \in B'$. This in turn gives $p(x)|\text{gcd}(g(x'),h(x'))$. Thus $p(x)$ must be a unit in $B$. Hence $\text{Pf}_{\bar{1}}(X)$ is an irreducible element in $B$. This implies $\text{Pf}_{\bar{1}}(X)$ is a prime element in $B$.

    Now, assume $\overline{\text{Pf}_{\bar{1}}(X) \alpha}=\overline{\beta \gamma}$ in $S/L$ for some $\alpha, \beta, \gamma \in S$. From Equation (\ref{pfaffians_rel}), one can assume that $\alpha \in B$. Since $\text{Pf}_{\bar{1}}(X)\alpha \in B$ and $\text{Pf}_{\bar{1}}(X) \alpha-\beta \gamma$ is a bihomogeneous element in $L=\langle g_1, \ldots, g_{n-1}, y_n \rangle \subseteq S$, one can assume that $ \beta \gamma \in B$ and hence $\text{Pf}_{\bar{1}}(X) \alpha-\beta \gamma=0$ in $S$. That is, $\text{Pf}_{\bar{1}}(X) | \beta$ or $\text{Pf}_{\bar{1}}(X) | \gamma$, since $\text{Pf}_{\bar{1}}(X)$ is a prime element in $B$. Thus, $\overline{\text{Pf}_{\bar{1}}(X)}| \overline{\beta}$ or $\overline{\text{Pf}_{\bar{1}}(X)}| \overline{\gamma}$ and hence, $\overline{\text{Pf}_{\bar{1}}(X)}$ is a prime element in $S/L$.
    
    \item From $(3)$ of this lemma, we have the inclusion $\supseteq$. Moreover, since $J$ is a prime ideal containing $J'$, we get the other inclusion.
\end{enumerate}
\end{proof}

The properties mentioned in Lemma \ref{technical_lemma} are found to be interestingly similar to those satisfied by the defining relations of the Rees algebra of an ideal generated by a regular sequence ( \cite[Lemma 3.2]{Neeraj}).

\vspace{1mm}

\begin{prop}
Let $B=K[X]$ and $I=\text{Pf}_{n-1}(X)=\text{Pf}_{2r}(X).$ Then the defining relations of $\mathcal{R}(I)$ given in Remark $\ref{generic_rees}(iii)$
form a $d$-sequence. Moreover, they generate an almost complete intersection ideal. 
\end{prop}

\begin{proof}

It is known that the Rees algebra of $I$ is Cohen-Macaulay \cite[Proposition 2.8]{DECH}. This implies height of the defining ideal of the Rees algebra is $n-1$. From Lemma \ref{technical_lemma}$(1)$ we have, $ \{ g_1, \ldots,g_{n-1} \} $ forms a regular sequence. To prove that $\{ g_1, \ldots,g_n \}$ forms a $d$-sequence, it suffices to show that $(J':g_n)=(J':{g_n}^2)$ where $J' = \langle g_1, \ldots,g_{n-1} \rangle$. Clearly $(J':g_n) \subseteq (J':{g_n}^2)$. To see the other inclusion, let $\alpha \in (J':{g_n}^2)$ $i.e.,$ $\alpha \cdot g_n \in (J':g_n)$. Since $(J':g_n)$ is a prime ideal (from the proof of Lemma \ref{technical_lemma}$(2)$) and $g_n \notin (J':g_n)$, this implies $\alpha \in (J':g_n)$. Hence, the equality follows.

Since the defining ideal is minimally generated by $n$ elements, one more than the height of the ideal, and the sequence forms a $d$-sequence, by \cite[Lemma 4.2]{a.c.i}, they generate an almost complete intersection ideal.

\end{proof}

\vspace{2mm}

It is evident from Remark $\ref{generic_rees}(iii)$ that the defining relations of the Rees algebra of the maximal order Pfaffians of the generic skew-symmetric matrix are generated by quadrics. So, it is natural to ask if something stronger holds. In this direction, we pose the following question.

\vspace{1mm}

\begin{quest}
Is it true that the Rees algebra of the maximal order Pfaffians in the generic case is always Koszul? 
\end{quest}

The following study indicates that it is not true in general.

\vspace{2mm}

Let $X_1$ be the generic skew-symmetric matrix of order $3$ and consider $B_{X_1}=K[X_1]$ with $I_1=\text{Pf}_2(X_1) = \langle x_{12},x_{13},x_{23} \rangle $ which is the graded maximal ideal of $B_{X_1}$. Then the defining ideal of $\mathcal{R}(I_1)$ is given by the $2$-minors of the matrix $ \begin{pmatrix}
x_{12} & x_{13} & x_{23}\\
y_1 & y_2 & y_3
\end{pmatrix}$ and the minimal bigraded $S$-free resolution of $\mathcal{R}(I_1)$ has the form,
$$ 0 \longrightarrow \begin{matrix}
                       S(-1,-2) \\
                        \oplus  \\
                       S(-2,-1)
                      \end{matrix} 
      \longrightarrow \begin{matrix}
                       S(-1,-1)^3 \\
                      \end{matrix} 
       \longrightarrow S \longrightarrow 0. $$

Since $\reg_S(\mathcal{R}(I_1))=1$ and $S$ seen as a graded ring with respect to the total degree is Koszul, by transfer of Koszulness (\cite[Theorem $2$]{CNR}), $\mathcal{R}(I_1)$ is Koszul. In fact, the defining ideal of $\mathcal{R}(I_1)$ is observed to be generated by a Gr$\ddot{\textrm{o}}$bner basis of quadrics with respect to the graded reverse lexicographic order induced by $x_{12} > x_{13} > x_{23}>y_1 > y_2 > y_3$.

Now, let $X_2$ be the generic skew-symmetric matrix of order $5$. Consider $B_{X_2}=K[X_2]$ with $I_2=\text{Pf}_4(X_2) =   \langle x_{14}x_{23} - x_{13}x_{24} + x_{12}x_{34}, \; x_{15}x_{23} - x_{13}x_{25} + x_{12}x_{35}, \; x_{15}x_{24} - x_{14}x_{25} + x_{12}x_{45}, \; x_{15}x_{34} - x_{14}x_{35} + x_{13}x_{45}, \; x_{25}x_{34} - x_{24}x_{35} + x_{23}x_{45} \rangle $. The following Betti table of the Pfaffian ring $B_{X_2}/I_2$ shows that $I_2$ does not have a linear resolution.

\vspace{1mm}

$$\beta_{(i,i+j)}= \;
\begin{tabular}{c| c c c c}
        & \red{0} & \red{1} & \red{2} & \red{3} \\
        \hline
\red{0} &    1    &    0   &     0    & 0 \\
\red{1} &    0    &    5   &     5    & 0 \\
\red{2} &    0    &    0   &     0    & 1 \\
\end{tabular}
$$ 

\vspace{1mm}

As a consequence, $\mathcal{R}(I_2)$ as a graded ring with respect to the total degree is not Koszul (Remark \ref{Koszulness_linear} (iii)).

\vspace{2mm}

In the following proposition, consider $I_1$ and $I_2$ as defined before. 

\begin{prop} \label{generic} Let $X$ be a generic skew-symmetric matrix of odd order $n$ and $B=K[X]$ with $I=\text{Pf}_{n-1}(X)$ for $n=3,5$. Then,
\begin{enumerate}[a)]
    \item $\mathcal{R}(I_1)$ is Koszul, whereas $\mathcal{R}(I_2)$ is not Koszul.
    \item $I^j$ has a linear resolution for all $j \geq 2$.
    \item reg$_x^S(\mathcal{R}(I^j))=0$ for $j \geq 2$.
     
\end{enumerate}
\end{prop}

\begin{proof} 
\begin{enumerate}[a)]
\item Koszulness of $\mathcal{R}(I_1)$ and non-Koszulness of $\mathcal{R}(I_2)$, follows from the preceding discussion.

\item For $n=3$, since $\mathcal{R}(I_1)$ is Koszul, we have in fact reg$_B (I_1)^j=2j$ for $j \geq 1$. For $n=5$, from Macaulay2 computations one obtains, reg$_A(I^{j})=2j$ for $2 \leq j \leq 6$. The result then follows from the result of Cutkosky, Herzog and Trung, which says that if an ideal $I$ of a ring $A$ is generated by a $d$-sequence of $n$ forms and is equigenerated of degree $r$, then for all  $j \geq n+1$, reg$_A(I^j)=(j-n-1)r+$reg$_A(I^{n+1})$ \cite[Corollary 3.8]{Cutkosky&Herzog&Trung}.

\item Follows from b) and \cite[Theorem 5.2]{x-reg}.

\end{enumerate}

\end{proof}

\vspace{1mm}

\begin{rem}  
For $\Delta=(1,1)$, the following is a realization of the diagonal subalgebra of the Rees algebra $\mathcal{R}(I_1)$, as a quotient of the standard graded polynomial ring. 
\vspace{2mm}

For $n=3$, let $T=\begin{pmatrix}
                                             t_{11} & t_{12} & t_{13} \\
                                            t_{21} & t_{22} & t_{23} \\
                                            t_{31} & t_{32} & t_{33} \\
                                            \end{pmatrix}$
where for $1 \leq j \leq 3$, $t_{1j}=x_{12}y_j$, $t_{2j}=x_{13}y_j$ and $t_{3j}=x_{23}y_j$.
Then $\mathcal{R}(I_1)_{\Delta}$ can be identified with $K[T]/I_2(T) +\langle t_{12}-t_{23},\; t_{21}-t_{32},\; t_{11}-t_{33} \rangle$ where $I_2(T)$ denotes the ideal generated by the $2$-minors of $T$ (from Equation (\ref{presentation}) and Lemma \ref{present2}).
From the above defining relations, we get 
$\mathcal{R}(I_1)_{\Delta}\cong K[T]/ \langle - t_{12}t_{21} + t_{11}t_{22},\; t_{11}t_{21} - t_{12}t_{31}, \; t_{21}^2 - t_{22}t_{31},\; t_{11}t_{12} - t_{13}t_{21}, \; t_{11}^2  - t_{13}t_{31}, \; t_{11}t_{21} - t_{12}t_{31}, \; t_{12}^2  - t_{13}t_{22}, \; t_{11}t_{12} - t_{13}t_{21}, \; - t_{12}t_{21} + t_{11}t_{22}\rangle.$ 

Similarly, we can write an expression for $\mathcal{R}(I_2)$ for $\Delta=(1,1)$.

\end{rem}

\vspace{1mm}

\vspace{1mm}

It is evident from Proposition \ref{generic} that $\mathcal{R}(I)$ is not always Koszul for the maximal order Pfaffian ideal $I$ of the general skew-symmetric matrix $X$. Regardless, the following result shows that, all its diagonals are always Koszul.

\vspace{1mm}

\begin{thm} \label{Koszul_generic_diagonals}
Let $B=K[X]$ and $I=$Pf$_{n-1}(X)$ where $n=2r+1$, $r \in \NN$. Then $\mathcal{R}(I)_{\Delta}$ is Koszul for all $\Delta=(c,e)$, $c >0$, $e>0$.
\end{thm}

\begin{proof}

Let $T=S/J'$ where $J'=\langle g_1, \ldots, g_{n-1} \rangle$. Then  $T_{\Delta}$ is Koszul for all $\Delta$ with $c > 0$ and $e >0$ (\cite[Proposition 2.10]{Neeraj}(i) which holds for all $c > 0$). Thus to prove the theorem, by \cite[Lemma $6.6$]{CHTV}, it suffices to show that reg$_{T_{\Delta}}\mathcal{R}(I)_{\Delta} \leq 1$  since $\mathcal{R}(I)=T/g_nT$.
To this end, we only show some of the steps explicitly since the proof is identical to  \cite[Theorem 3.1]{Neeraj} and \cite[Corollary 3.3]{CC}. To see the steps in detail, refer to the cited results.

From the assumptions, since $g_ny_n=0$ in $T$, we can consider the following complex.

\begin{equation} \label{complex}
\mathbb{F.}:  \cdots \stackrel{y_n}\longrightarrow T(-2r,-3) \stackrel{g_n} \longrightarrow T(-r,-2) \stackrel{y_n}\longrightarrow T(-r,-1) \stackrel{g_n} \longrightarrow T \longrightarrow 0.     
\end{equation}

Then the homology of $\mathbb{F}$ can be seen as follows: 

$\text{H}_k(\mathbb{F})= \left \{ \begin{array}{ll}
\mathcal{R}(I)           & \text{ if }k=0 \\
0                    & \text{ if }k=2i \text{ and }i>0 \\
   \bigl[ S / {\langle  y_1, \ldots, y_n \rangle} \bigr] (-(i+1)r-r,-2i-1) & \text{ if }k=2i+1 \text{ and }i \geq 0 
    \end{array}
\right.
$

Clearly $H_{2i}(\mathbb{F}_{\Delta})=0$ and $H_{0}(\mathbb{F}_{\Delta})=\mathcal{R}(I)_{\Delta}$. Taking the $j$-th degree component of $H_{2i+1}(\mathbb{F}_{\Delta})$ we get,
$$ (H_{2i+1}(\mathbb{F})_{\Delta})_j= [\, S/ \langle y_1, \ldots, y_n \rangle ]\,(-(i+2)r+jc,-2i-1+je).$$ Applying \cite[Lemma 2.7]{Neeraj}, we have,
$$ \text{reg}_{T_{\Delta}}(\mathcal{R}(I)_{\Delta}) \leq \text{sup}\{ \mu, \nu \}.$$
where $\mu= \text{sup}\{ \text{reg}_{T_{\Delta}}(F_k)_{\Delta} -k: k \geq 0 \}$, $F_k$ being the $k$-th module in the Complex (\ref{complex}) and
$\nu=\text{sup}\{ \text{reg}_{T_{\Delta}}H_k(\mathbb{F}_\Delta) -(k+1): k \geq 1 \}$.

Let $c > 0$ and $e >0$.
Then we have the following:
\begin{enumerate}
    \item $\mu \leq 1$. 
    \item Taking the $j$-th degree component of $H_{2i+1}(\mathbb{F})_{\Delta}$, it is observed that $(H_{2i+1}(\mathbb{F})_{\Delta})_j=0$ if $-(2i+1)+j \geq 1 $, that is, if $j \geq 2i+2$. So the largest non-zero degree component of $H_{2i+1}(\mathbb{F}_{\Delta})$ is atmost $2i+1$. Since $H_{2i+1}(\mathbb{F}_{\Delta})=H_{2i+1}(\mathbb{F})_{\Delta}$, by \cite[Equation (1)]{Neeraj} we get reg$_{T_{\Delta}}H_k(\mathbb{F}_{\Delta}) \leq$reg$_{S'}H_k(\mathbb{F}_{\Delta}) \leq k$
 where $S'$ is a polynomial ring such that there exists a surjection from $S'$ onto $T_{\Delta}$. Hence $\nu \leq -1$.
\end{enumerate}
Therefore $\mathcal{R}(I)_{\Delta}$ is Koszul for $c>0$ and $e >0$.

\end{proof}

It is known that the Rees algebra of the maximal generic Pfaffians is Cohen-Macaulay.
Following are some observations made regarding the Cohen-Macaulayness of the diagonals of the same.

\begin{rem}
Let $X$ be the generic skew-symmetric matrix of odd order $n$ and $B=K[X]$ with $I=\text{Pf}_{n-1}(X)$.

\begin{enumerate}
    \item Then $\mathcal{R}(I)_{\Delta}$ is Cohen-Macaulay for $c \gg 0$ and $e>0$.
    This can be seen as a consequence of $I$ being generated by a $d$-sequence which implies vanishing of y-regularity of $\mathcal{R}(I)$ \cite[Corollary 3.2]{Rom1} and 
    \cite[Corollary 3.14]{CHTV}.
    \item In particular, when  $n=3,5$, $\mathcal{R}(I)_{\Delta}$ is Cohen-Macaulay if $c \geq 2e+1$, $c,e >0$ which is a consequence of the repeated application of \cite[Corollary 18.6]{Eisenbud_text}, Lemma \ref{CM_criteria} and Lemma \ref{dimension}.  
\end{enumerate}
\end{rem}

\section{Pfaffian ideals of sparse skew-symmetric matrices} \label{sparse_skew}

For a generic skew-symmetric matrix $X$ of odd order $n$, computations suggest that as $n$ increases, the number of generators for $I=\text{Pf}_{n-1}(X)$ increases largely. Moreover, the expression for the generators of $I$ becomes too complex, thereby making the study of the minimal graded free resolutions of the corresponding Rees algebras hard.
Thus, for the generators of a Pfaffian ideal and the defining ideal of the Rees algebra to satisfy some properties, it makes sense to focus on the sparse form of skew-symmetric matrices.

Note that we primarily focus on the maximal order Pfaffians since an ideal $I$ generated by the non-maximal order Pfaffians most often leads to the defining relations of the Rees algebra being of total degree greater than $2$. For example,

    {\footnotesize{
    $$X= \begin{bmatrix}
                          0      & x_{12} &   0    &   x_{14}      &0 &0 &0\\
                       -x_{12}   &   0    & x_{23} &   0           &x_{25}  
                       & 0 & 0 \\
                          0      &-x_{23} &   0    & x_{34}        &0 & x_{36} &0 \\
                       -x_{14}   &   0    &-x_{34} &   0           &x_{45}& 0 & x_{47}\\
                          0      &-x_{25} &   0    & -x_{45}       &0 & x_{56} &0   \\
                          0 & 0     &-x_{36} &   0    & -x_{56}       &0 & x_{67}  \\
                          0      &0 &   0    & -x_{47}       &0 & -x_{67} &0   \\
                      \end{bmatrix}_{7 \times 7}$$\\
                      }}

In the above case, for $I=$Pf$_4(X)$ we get, $\mathcal{R}(I)$ to have $52$ generators of bidegree $(1,1)$, $14$ generators of bidegree $(0,2)$ and $3$ generators of bidegree $(0,3)$. Hence, it is clearly not generated by quadrics (with respect to total degree). 

\vspace{2mm}

In an attempt to study the sparse skew-symmetric matrices, we first focus on the Pfaffians of the tridiagonal matrices of the form given in Theorem \ref{tridiagonal}. Following are some of the terms which will be used for the same. 

\vspace{1mm}
A homogeneous ideal $I$ of a standard graded ring $A$ is of Gr\"obner linear type if the ideal is of linear type with the linear relations of the defining ideal of $\mathcal{R}(I) \cong K[X,Y]/J$ forming a Gr\"obner basis with respect to some monomial order on $K[X,Y]$. 
Let a sequence of monomials $m_1, \ldots, m_s$ in the set of indeterminates $x_1, \ldots,x_n$ be denoted as $m_i=x_1^{a_{i_1}} \cdots x_n^{a_{i_n}}$ with $a_{i_1} >0, \ldots, a_{i_n}>0$. Then such a sequence of monomials is said to form an \textit{$M$-sequence} if, for all $1 \leq i \leq s$, there exists a total order on the set of indeterminates, say $x_1 < \cdots < x_n$, such that whenever $x_k | m_j$ with $1 \leq k \leq n$ and $i<j$, $x_k^{a_{i_k}}\cdots x_n^{a_{i_n}}|m_j$ (cf. \cite{CN}). The authors in \cite[Theorem 2.4 (i)]{CN} prove that the ideals generated by $M$-sequences are of Gr\"obner linear type.
For a monomial $m$ and an indeterminate $x$, let $O_x(m)$ denote the exponent of $x$ in $m$. Then a sequence of monomials $m_1, \ldots, m_s$ in the set of indeterminates $X$ is said to be of \textit{interval type} if for all $1 \leq i <j \leq s$ and $x|\text{gcd}(m_i,m_j)$, one has $O_x(m_i) \leq O_x(m_k)$ for all $i \leq k \leq j$ (cf. \cite{CN}). It is known that a sequence of interval type implies an $M$-sequence \cite[Proposition 3.2]{CN}.

\vspace{1.5mm}

\begin{lem} \label{determinant}
Let $\footnotesize{X= \begin{bmatrix}
                          0      & x_1 &   0    &   0      & \ldots&0  &0 \\
                       -x_1   &   0    & x_2 &   0      & \ldots &0 &0 \\
                          0      &-x_2 &   0    & x_3   & \ldots &0 &0 \\
                          0      &   0    &-x_3 &   0      & \ldots & 0  &0\\
                        \vdots   & \vdots &\vdots  & \vdots   & \ddots  &\vdots  & \vdots  \\
                          0      &   0    &   0    &0& \ldots  & 0    &x_{n-1} \\
                          0      &   0    &   0    &   0  & \ldots &-x_{n-1}&0   \\
                      \end{bmatrix}_{n \times n}}$,
where $n \in \NN$.

Then,

\begin{enumerate}
    \item det $X=0$, if $n$ is odd.
    \item det $X= \prod_{i=2k+1}{x_i}^2$, $k=0,\ldots,\dfrac{n-2}{2}$, if $n$ is even.
\end{enumerate}
\end{lem}

\begin{proof}

Let $X$ be a matrix of the above form.
\begin{enumerate}
    \item Follows from the property of the matrix being skew-symmetric.
    \item The proof follows by induction on n. Clearly, for $n=2,4$, the determinant is given by $x_1^2$ and $x_1^2x_3^2$, respectively. Thus the statement holds in these cases. Let $n$ be an even integer such that $n>4$, then the Laplace expansion of the determinant along the last column and then the last but one column gives det $X=(-x_{n-1})(-x_{n-1})$det $X'$ where $X'$ is a matrix of order $n-2$ whose determinant is given by the induction hypothesis. Thus, the required result is obtained.
\end{enumerate}
\end{proof}

\vspace{1mm}

\begin{thm}\label{tridiagonal}

Let $\footnotesize{X_3= \begin{bmatrix}
                          0      & x_{12} &   0    &   0      & \ldots& 0  &0 \\ 
                       -x_{12}   &   0    & x_{23} &   0      & \ldots & 0  &0 \\
                          0      &-x_{23} &   0    & x_{34}   & \ldots & 0  &0 \\
                         0      &   0    &-x_{34} &   0      & \ldots &0  &0\\
                        \vdots   & \vdots &\vdots  & \vdots   & \ddots  & \vdots & \vdots   \\
                          0      &   0    &   0    &0 &  \ldots &  0    &x_{n-1 \; n} \\
                          0      &   0    &   0    &   0 & \ldots & -x_{n-1 \; n}&0   \\
                      \end{bmatrix}_{n \times n}}$\\
\vspace{4mm}
\noindent where $n=2r+1,\, r \in \NN \cup \{ 0\}$. Then $I=Pf_{n-1}(X_3)$ is a monomial ideal of Gr\"obner linear type.  
\end{thm}

\begin{note}
In the case of a generic skew-symmetric matrix, it is known from \cite[Theorem 2.2]{Baetica} that the Pfaffian ideal $I=$Pf$_{n-1}(X)$ is of linear type. However results in that direction cannot be directly applied in sparse cases like the ones above. Hence we will separately show that the Pfaffian ideals, in this case, are of linear type. 
\end{note}

\begin{proof}
Using Lemma \ref{determinant}, we get the generators of the Pfaffian ideal $I=\text{Pf}_{n-1}(X_3)$ to have the following form, 
\begin{equation} \label{tridiagonal_paffian}
p_1= \prod_{k=1}^r x_{2k\;2k+1},\text{ }p_i=x_{12}\prod_{j<i}x_{j \; j+1} \prod_{i<k \leq 2r} x_{k \; k+1}, 
\end{equation}
where   $i=2p+1,\, 1 \leq p \leq r ; \; 
j=1+2\ell,\,\ell\geq 1; \; k=i+1+2m, \, m \geq 0$.
That is, we have $ \{ p_i: i=2p+1, \, 0 \leq p \leq r \}$ to be the generating set of the Pfaffian ideal $I=$Pf$_{n-1}(X_3).$ 

Now for $1 \leq k < m \leq r+1$, assume that $x_{j\;j+1}|\text{gcd}(p_k,p_m)$ for $1 \leq j \leq n-1$. Then there are two possibilities.
\begin{enumerate}
    \item $j < k.$\\
    This implies that $j$ is odd and $j <\ell$ for all $k \leq\ell\leq m$. Thus $x_{j \; j+1}|p_l$.
    \item $j >k.$\\
    Then we have $j$ to be even. Hence $j >m$ and so $j>\ell$ for all $k \leq\ell\leq m$. Thus $x_{j \; j+1}|p_{\ell}$. 
\end{enumerate}

Since the generators are squarefree monomials, this implies that $\{ p_i: i=2p+1, \, 1 \leq p \leq r \}$ forms a sequence of interval type. In particular, it forms an $M$-sequence with respect to some total order and thus is of Gr\"obner linear type.

\end{proof}

\begin{note}
For $n,m \in \NN$, let $(n \; m)$ be the pair of integers which denotes the indices of the indeterminates as entries in the matrix $X_3$.
\end{note}

\begin{cor} \label{vertex_cover_ideal}
The ideal of Pfaffians Pf$_{n-1}(X_3)$ can be seen as the vertex cover ideal of an unmixed bipartite graph.
\end{cor}

\begin{proof}
Let $G$ be the following bipartite graph.

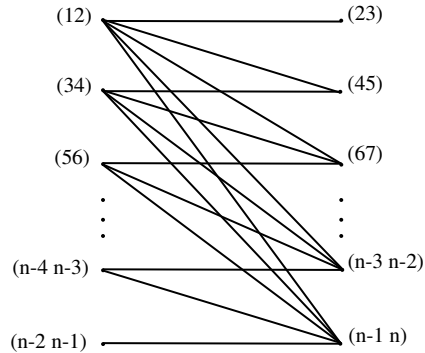
\begin{center}

\tikzset{every picture/.style={line width=0.75pt}} 

\begin{tikzpicture}[x=0.75pt,y=0.75pt,yscale=-1,xscale=1]

\draw    (225.68,19.3) -- (344.55,19.6) ;
\draw    (225.68,54.61) -- (344.55,54.91) ;
\draw    (225.68,90.83) -- (344.55,91.13) ;
\draw    (225.68,144.25) -- (344.55,144.56) ;
\draw   (226.14,18.82) .. controls (226.15,19.08) and (225.94,19.3) .. (225.68,19.3) .. controls (225.41,19.3) and (225.19,19.09) .. (225.18,18.82) .. controls (225.17,18.56) and (225.38,18.35) .. (225.65,18.35) .. controls (225.91,18.35) and (226.14,18.56) .. (226.14,18.82) -- cycle ;
\draw   (345.01,19.12) .. controls (345.02,19.38) and (344.81,19.6) .. (344.55,19.6) .. controls (344.28,19.6) and (344.06,19.39) .. (344.05,19.13) .. controls (344.04,18.86) and (344.25,18.65) .. (344.52,18.65) .. controls (344.78,18.65) and (345,18.86) .. (345.01,19.12) -- cycle ;
\draw   (226.14,54.13) .. controls (226.15,54.4) and (225.94,54.61) .. (225.68,54.61) .. controls (225.41,54.61) and (225.19,54.4) .. (225.18,54.14) .. controls (225.17,53.88) and (225.38,53.66) .. (225.65,53.66) .. controls (225.91,53.66) and (226.14,53.87) .. (226.14,54.13) -- cycle ;
\draw   (344.55,54.91) .. controls (344.56,55.18) and (344.35,55.39) .. (344.08,55.39) .. controls (343.82,55.39) and (343.59,55.18) .. (343.59,54.92) .. controls (343.58,54.66) and (343.79,54.44) .. (344.05,54.44) .. controls (344.32,54.44) and (344.54,54.65) .. (344.55,54.91) -- cycle ;
\draw   (226.18,91.3) .. controls (226.18,91.57) and (225.98,91.78) .. (225.71,91.78) .. controls (225.44,91.78) and (225.22,91.57) .. (225.21,91.31) .. controls (225.21,91.05) and (225.41,90.83) .. (225.68,90.83) .. controls (225.94,90.83) and (226.17,91.04) .. (226.18,91.3) -- cycle ;
\draw   (345.04,91.61) .. controls (345.05,91.87) and (344.84,92.08) .. (344.58,92.08) .. controls (344.31,92.09) and (344.09,91.87) .. (344.08,91.61) .. controls (344.07,91.35) and (344.28,91.13) .. (344.55,91.13) .. controls (344.81,91.13) and (345.03,91.34) .. (345.04,91.61) -- cycle ;
\draw   (226.14,109.37) .. controls (226.15,109.63) and (225.94,109.84) .. (225.68,109.85) .. controls (225.41,109.85) and (225.19,109.63) .. (225.18,109.37) .. controls (225.17,109.11) and (225.38,108.89) .. (225.65,108.89) .. controls (225.91,108.89) and (226.14,109.1) .. (226.14,109.37) -- cycle ;
\draw   (226.14,119.33) .. controls (226.15,119.59) and (225.94,119.8) .. (225.68,119.81) .. controls (225.41,119.81) and (225.19,119.6) .. (225.18,119.33) .. controls (225.17,119.07) and (225.38,118.86) .. (225.65,118.85) .. controls (225.91,118.85) and (226.14,119.06) .. (226.14,119.33) -- cycle ;
\draw   (226.14,127.48) .. controls (226.15,127.74) and (225.94,127.95) .. (225.68,127.96) .. controls (225.41,127.96) and (225.19,127.74) .. (225.18,127.48) .. controls (225.17,127.22) and (225.38,127) .. (225.65,127) .. controls (225.91,127) and (226.14,127.21) .. (226.14,127.48) -- cycle ;
\draw   (345.01,109.37) .. controls (345.02,109.63) and (344.81,109.84) .. (344.55,109.85) .. controls (344.28,109.85) and (344.06,109.63) .. (344.05,109.37) .. controls (344.04,109.11) and (344.25,108.89) .. (344.52,108.89) .. controls (344.78,108.89) and (345,109.1) .. (345.01,109.37) -- cycle ;
\draw   (345.01,119.33) .. controls (345.02,119.59) and (344.81,119.8) .. (344.55,119.81) .. controls (344.28,119.81) and (344.06,119.6) .. (344.05,119.33) .. controls (344.04,119.07) and (344.25,118.86) .. (344.52,118.85) .. controls (344.78,118.85) and (345,119.06) .. (345.01,119.33) -- cycle ;
\draw   (345.01,127.48) .. controls (345.02,127.74) and (344.81,127.95) .. (344.55,127.96) .. controls (344.28,127.96) and (344.06,127.74) .. (344.05,127.48) .. controls (344.04,127.22) and (344.25,127) .. (344.52,127) .. controls (344.78,127) and (345,127.21) .. (345.01,127.48) -- cycle ;
\draw    (225.68,181.38) -- (344.55,181.68) ;
\draw   (226.18,144.73) .. controls (226.18,144.99) and (225.98,145.2) .. (225.71,145.21) .. controls (225.44,145.21) and (225.22,144.99) .. (225.21,144.73) .. controls (225.21,144.47) and (225.41,144.25) .. (225.68,144.25) .. controls (225.94,144.25) and (226.17,144.46) .. (226.18,144.73) -- cycle ;
\draw   (226.18,181.85) .. controls (226.18,182.11) and (225.98,182.33) .. (225.71,182.33) .. controls (225.44,182.33) and (225.22,182.12) .. (225.21,181.86) .. controls (225.21,181.59) and (225.41,181.38) .. (225.68,181.38) .. controls (225.94,181.38) and (226.17,181.59) .. (226.18,181.85) -- cycle ;
\draw   (345.51,144.55) .. controls (345.52,144.81) and (345.31,145.03) .. (345.04,145.03) .. controls (344.78,145.03) and (344.56,144.82) .. (344.55,144.56) .. controls (344.54,144.29) and (344.75,144.08) .. (345.01,144.08) .. controls (345.28,144.08) and (345.5,144.29) .. (345.51,144.55) -- cycle ;
\draw   (344.55,181.68) .. controls (344.56,181.94) and (344.35,182.16) .. (344.08,182.16) .. controls (343.82,182.16) and (343.59,181.95) .. (343.59,181.68) .. controls (343.58,181.42) and (343.79,181.21) .. (344.05,181.21) .. controls (344.32,181.2) and (344.54,181.42) .. (344.55,181.68) -- cycle ;
\draw    (226.14,18.82) -- (344.08,55.39) ;
\draw    (225.68,19.3) -- (344.55,91.13) ;
\draw    (225.18,18.82) -- (345.51,144.55) ;
\draw    (225.68,19.3) -- (344.05,181.21) ;
\draw    (225.18,54.14) -- (344.55,91.13) ;
\draw    (225.68,54.61) -- (345.51,144.55) ;
\draw    (225.68,54.61) -- (344.55,181.68) ;
\draw    (225.21,91.31) -- (345.51,144.55) ;
\draw    (225.71,91.78) -- (344.55,181.68) ;
\draw    (225.21,144.73) -- (344.05,181.21) ;

\draw (201.01,10.91) node [anchor=north west][inner sep=0.75pt]   [align=left] {{\scriptsize (12)}};
\draw (346.22,10.01) node [anchor=north west][inner sep=0.75pt]   [align=left] {{\scriptsize (23)}};
\draw (346.22,45.23) node [anchor=north west][inner sep=0.75pt]   [align=left] {{\scriptsize (45)}};
\draw (346.22,80.35) node [anchor=north west][inner sep=0.75pt]   [align=left] {{\scriptsize (67)}};
\draw (346.71,171.99) node [anchor=north west][inner sep=0.75pt]   [align=left] {{\scriptsize (n-1 n)}};
\draw (346.28,134.87) node [anchor=north west][inner sep=0.75pt]   [align=left] {{\scriptsize (n-3 n-2)}};
\draw (201.01,46.13) node [anchor=north west][inner sep=0.75pt]   [align=left] {{\scriptsize (34)}};
\draw (200.01,82.45) node [anchor=north west][inner sep=0.75pt]   [align=left] {{\scriptsize (56)}};
\draw (179.03,137.77) node [anchor=north west][inner sep=0.75pt]   [align=left] {{\scriptsize (n-4 n-3)}};
\draw (178.03,175.09) node [anchor=north west][inner sep=0.75pt]   [align=left] {{\scriptsize (n-2 n-1)}};

\end{tikzpicture}

\end{center}

Then it suffices to show that the set of all minimal vertex covers of $G$ are given by $\{ (2k\;2k+1) \}$, $1 \leq k \leq r$ and $\{ (1\,2),(j\;j+1),(k\;k+1) \}$, where $i=2p+1,\, 1 \leq p \leq r ; \; 
j=1+2l,\,\ell\geq 1; \; k=i+1+2m, \, m \geq 0$.
We prove this by induction on $r$ where the order of the matrix is $n=2r+1$.
For $r=1$ and $r=2$, the minimal vertex covers are given by $\{(1\,2)\},\{ (2\,3) \}$ and $\{ (1\,2),(3\,4)\}$, $\{ (1\,2),(4\,5)\}$, $\{ (2\,3),(4\,5)\}$ respectively. Then the result is true for the base cases.
Now let $r>2$. By induction hypothesis, the minimal vertex covers of $G$ on $2(r-1)$ vertices will have the form $\{ (2k\;2k+1) \}$,  $1 \leq k \leq r-1$ and $i=2p+1,\, 1 \leq p \leq r-1 ; \; 
j=1+2\ell,\,\ell\geq 1; \; k=i+1+2m, \, m \geq 0$. Observe that, only one vertex cover say $u= \{(\ell\;\ell+1)\}$, $\ell=2m+1$, $0 \leq m \leq r-2$ contains $(2r-3\;2r-2)$ and the rest of them contains $(2r-2\;2r-1)$. Then $u$ can be extended  to form a vertex cover of $G$ on $2r$ vertices by adjoining either $(2r-1\;2r)$ or $(2r\;2r+1)$ whereas,  the other vertex covers containing $(2r-2\;2r-1)$ can be extended only by adding $(2r\;2r+1)$. This is because if instead $(2r-1\;2r)$ is added, the edge between $(2r-3\;2r-2)$ and $(2r \;2r+1)$ would have an empty intersection with the set. Thus the minimal vertex covers of $G$ on $n=2r+1$ vertices is given by $\{ (2k\;2k+1) \}$, $1 \leq  k \leq r$ and $\{ (1\,2),(j\;j+1),(k\;k+1) \}$, where  
$j=1+2\ell,\,\ell\geq 1; \; i=2p+1,\, 1 \leq p \leq r ; \; k=i+1+2m, \, m \geq 0$.

\noindent Thus, the ideal of Pfaffians Pf$_{n-1}(X_3)$ in Theorem \ref{tridiagonal} can be seen as the vertex cover ideal of the above unmixed bipartite graph $G$.

\end{proof}

\vspace{1mm}

 As a consequence of viewing the Pfaffian ideal as a vertex cover ideal, the following can be said regarding its ordinary and symbolic powers. Note that, here $I^{(s)}$ denotes the $s\text{-th}$-symbolic power of the ideal $I$ and for a monomial ideal $I$, with a linear order $m_1 \prec m_2 \prec \cdots \prec m_n$ on the minimal generators of $I$, it is said to have linear quotients with respect to $\prec$, if for all $1 \leq i \leq n-1$, $(\langle m_1, \ldots, m_i \rangle:m_{i+1})$ is generated by linear forms. It is proved that if $I$ is an equigenerated monomial ideal with linear quotients, then $I$ has a linear resolution \cite{Zheng}. For more details related to the graph-theoretic terms used in Corollary $\ref{powers}$, please refer \cite{SPOP}.

\begin{cor} \label{powers}
    For $I=$Pf$_{n-1}(X_3)$, $I^s$ and $I^{(s)}$ have linear quotients for all $s \geq 1$.
\end{cor}

\begin{proof}
   From Corollary \ref{vertex_cover_ideal}, $I$ is the vertex cover ideal of an unmixed bipartite graph $G$. In particular, it is the vertex cover ideal of a very well covered graph. Hence from \cite[Theorem 3.6]{SPOP}, $I^{(s)}$ has linear quotients for all $s \geq 1$. Moreover, since $G$ is a bipartite graph, from \cite[Corollary 2.6]{powers}, $I^s=I^{(s)}$ for all $s \geq 1$.
   
\end{proof}

\vspace{1.5mm}
\begin{thm}
Let $I=$Pf$_{n-1}(X_3)$. Then $\mathcal{R}(I)$ is Koszul and  Cohen-Macaulay.
\end{thm}

\begin{proof}
Let $\mathcal{R}(I) \cong S/J$ where $S=K[x_{12},\ldots,x_{n-1\; n},y_1,\ldots,y_{r+1}]$ and $J$ is the defining ideal of $\mathcal{R}(I)$. From the relation (\ref{D.Taylor}) in Section \ref{Preliminaries}, we get $J=\langle x_{i\;i+1}y_{j}-x_{i+1 \; i+2}y_{j+1}: \,i=2k+1, \, 0 \leq k \leq r-1, \, j= \frac{i+1}{2} \rangle$.
Then from \cite[Lemma 2.2]{reg-seq}, with respect to the graded reverse lexicographic term order induced by $x_{12} >x_{13}> \cdots   > x_{n-1\,n} > y_1 > y_2 > \cdots > y_{r+1}$, it can be concluded that $J$ is generated by a regular sequence and so is a complete intersection of quadrics. This implies $\mathcal{R}(I)$ is Cohen-Macaulay. Koszulness of $\mathcal{R}(I)$ follows from observing that the defining ideal of the Rees algebra is generated by a Gr\"obner basis of quadrics. 

\end{proof}

\vspace{1.5mm}

In \cite{KSSW}, the authors gave some bounds on $(c,e)$ for which the diagonal subalgebras of hypersurfaces become Cohen-Macaulay. In the following proposition, we try to extend it to the diagonals of algebras defined by  complete intersections.   
\vspace{1mm}

\begin{prop} \label{CM_diagonals}
Let $R=S/J$ be a standard bigraded $K$-algebra where $S=K[x_1,\ldots, x_n,y_1, \ldots, y_m]$ and $J$ a bihomogeneous ideal of $S$ generated by a regular sequence $\{ f_1, \ldots, f_\ell \}$ of bidegree $(a_i,b_i)$, $i =1, \ldots,\ell$. Let $d_1=max \{a_i: \, 1 \leq i \leq\ell \}$ and $d_2=max \{ b_i: \, 1 \leq i \leq\ell\}$.
Then for $c \geq -n+\ell d_1$ and $e \geq -m+\ell d_2$, $R_{\Delta}$ is Cohen-Macaulay for the following cases.

\begin{enumerate}
    \item For $\Delta=(c,e)$, if $R$ satisfies the property that all its associated primes do not contain $S_{(c,0)}$ and $S_{(0,e)}$. In particular, this is true if $R$ is a domain.
    \item If $R=R_1 \otimes_{K}R_2$ with $\text{dim}R=\text{dim} R_1+\text{dim} R_2$ where dim denotes Krull dimension.
\end{enumerate}

\end{prop}

\begin{proof}
We have that $R$ is defined by a complete intersection. Hence the Koszul complex on $\{f_1, \ldots, f_{\ell} \}$ resolves $R$. Applying $\Delta$ to the resolution of $R$ and then using \cite[Corollary 18.6]{Eisenbud_text} and Lemma  \ref{CM_criteria} gives depth($R_{\Delta}) \geq m+n-\ell-1$.
 \begin{enumerate}
  \item Assume for $\Delta=(c,e)$, $R$ satisfies the property that all its associated primes do not contain $S_{(c,0)}$ and $S_{(0,e)}$. Then the relative dimension of $R$ coincides with its Krull dimension \cite[Section 2.2]{SimisTrungValla}. Hence from the idea similar to \cite[Proposition 2.3]{SimisTrungValla}, it can be concluded that dim($R_{\Delta})=m+n-\ell-1$. 
  \item  Let $R=R_1 \otimes_{K}R_2$ with $\text{dim}R=\text{dim} R_1+\text{dim} R_2$. Then from \cite[Lemma 2.7]{CHTV}, dim$(R_{\Delta}) \leq m+n-\ell-1$ and from the lower bound on the depth of $R$ we obtain, dim$(R_{\Delta})=m+n-\ell-1$. 
 \end{enumerate}
\end{proof}

\vspace{1mm}

As an application of the previous results, we have the following observations.

\begin{thm}\label{tridiagonal3}
Let $I=Pf_{n-1}(X_3)$. Let $c,e$ be positive integers. Then,
\begin{enumerate}[(a)]
    \item $\mathcal{R}(I)_{\Delta}$ is Koszul for all $\Delta=(c,e)$. 
    \item $\mathcal{R}(I)_{\Delta}$ is Cohen-Macaulay for all $\Delta=(c,e)$.
\end{enumerate}
\end{thm}

\begin{proof}
Since $\mathcal{R}(I)$ is Koszul, Koszulness of $\mathcal{R}(I)_{\Delta}$ can be seen as a consequence of Remark \ref{Koszulness_linear} (ii).

\vspace{1mm}

We have that $\mathcal{R}(I)$ is a complete intersection domain. Hence the Cohen-Macaulayness of all its diagonals follow from Proposition \ref{CM_diagonals}
\end{proof}

\begin{rem}
\text{  }
\begin{enumerate}

    \item Let $I=$Pf$_{n-1}(X_3)$, $\Delta=(1,1)$ and $T=(t_{ij}), \, 1 \leq i \leq n-1, \, 1 \leq j \leq r+1$. Then $$\mathcal{R}(I)_{\Delta} \cong K[T]/I_2(T)+ \langle t_{ij}-t_{i+1j+1}: \, i=2k+1, \, 0 \leq k \leq r-1, \, j=\frac{i+1}{2} \rangle.$$

    \item In general, Pfaffian ideals need not be generated by monomials or binomials. For instance, for generic skew-symmetric matrices of order greater than $3$, the generators are seen to be neither monomials nor binomials. But, if the Pfaffian ideal is a squarefree monomial ideal, then it corresponds to a simplicial complex $\Delta'$, where the Pfaffian ideal can be identified as the Stanley Reisner ideal of $\Delta'$
    (\cite{Sturmfels}). 
    \item Theorem \ref{tridiagonal} is a special case where the Pfaffian ideal can be viewed as the vertex cover ideal of an unmixed bipartite graph.
    Note that the correspondence between skew-symmetric matrices and Pfaffian ideals is not unique.
For example, let $\footnotesize{X= \begin{bmatrix}
                          0      & x_{12} &   0    &   0      &0 \\
                       -x_{12}   &   0    & x_{23} &   x_{24} &0 \\
                          0      &-x_{23} &   0    & x_{34}   &0 \\
                          0      &-x_{24} &-x_{34} &   0      &x_{45}\\
                          0      &   0    &    0   & -x_{45}  & 0  \\
                       \end{bmatrix}_{5 \times 5}}$.
Then the maximal order Pfaffian ideal of the matrix $X$ coincides with Pf$_{4}(X_3)$.

\end{enumerate}
\end{rem}

\vspace{2mm}

Another form of a sparse skew-symmetric matrix that we have considered is the following.

\vspace{1mm}

\begin{prop} \label{pf_case2}
Let $n=2r+1, \, r \in \NN \cup \{ 0\} $ and let $X_4=\begin{bmatrix}
     O & A \\
     -A^T & C \\
     \end{bmatrix}$,
$O$ is a zero block matrix,
\[
A=\begin{bmatrix}
              x_{1\;r+2} & x_{1\;r+3} &\ldots & x_{1\;n-1}     &x_{1\;n}\\
              x_{2\;r+2} & x_{2\; r+3}  &\ldots & x_{2\; n-1}    &x_{2\;n}\\
              \vdots &\vdots &\vdots &\vdots &\vdots \\
              x_{r+1\;r+2} & x_{r+1\;r+3} &\ldots & x_{r+1\;n-1}   &x_{r+1\;n}\\
              \end{bmatrix}, \; \text{ and } \;
C=\begin{bmatrix}
               0 & x_{r+2\;r+3}  &\ldots  &x_{r+2\; n-1}       &x_{r+2\; n}\\
               -x_{r+2\;r+3} & 0 & \ldots  &x_{r+3\;n-1}       &x_{r+3\;n}\\
               \vdots &\vdots &\vdots & \vdots &\vdots \\
               -x_{r+2\;n} & -x_{r+3\;n} & \ldots     &-x_{n-1\;n}  &0 \\
               \end{bmatrix},
 \] 
 
 \vspace{1mm}
\noindent where $A$ is an $(r+1) \times r$ matrix and $C$ is a skew-symmetric matrix of order $r \times r$. Then,
\begin{enumerate}
    \item $I=Pf_{n-1}(X_4)$ is generated by an unconditioned $d$-sequence.
    \item $\mathcal{R}(I)$ is Koszul and Cohen-Macaulay.
    \item $I^j$ has a linear resolution for all $j \geq 1$.
\end{enumerate}

\end{prop}

\begin{proof}
\noindent Let $B=K[X_4]$. Let 
$X_4^{(i)}=\begin{bmatrix}

O^{(i)} & A^{(i)} \\
-(A^{(i)})^{T} & C^{(i)}  \\
\end{bmatrix}$ 
where $X_4^{(i)}$ is the submatrix of $X_4$ obtained by removing its $i\text{-th}$ row and the $i\text{-th}$ column. Since the Pfaffians of $X_4$ are the square roots of the determinants of $X_4^{(i)}$ for $i=1, \ldots, n$, it is clear that $C$ does not contribute to the Pfaffians of $X_4$. Thus, the generators of the Pfaffian ideal $I=$Pf$_{n-1}(X_4)$ is given by the $r$-minors of the first $r+1$ rows and the last $ (n-r-1)=r$ columns of $X_4$, which is $r+1 \choose r$ in number ($i.e.,$ contributed by the $r$-minors of the matrix $A$).  
\begin{enumerate}
    \item From \cite[Proposition 1.1]{CH1}, the Pfaffians form an unconditioned $d$-sequence.
    \item From example $1.2$ in \cite{CH1}, it follows that $\mathcal{R}(I) \cong S/J$ where
    $S=K[X_4,Y]$ with $Y=\begin{bmatrix}
   y_1 \cdots y_{r+1} \\
   \end{bmatrix}$ and the defining ideal $J$ has the following form, 
$$J=\langle g_i, i=1, \ldots, r \rangle  \text{ where } g_i= \sum_{k=1}^{r+1}(-1)^{k+1}x_{k \;n-(i-1)}y_k, \; i=1, \dots r.$$ Consider the following order on the indeterminates of $S$, $x_{1 \; n} > x_{2 \; n-1} > \cdots > x_{r \; n-(r-1)}$ followed by the other indeterminates, with the term order being graded lexicographic. Then from \cite[Lemma 2.2]{reg-seq}, $J$ is a complete intersection of quadrics. Thus being the complete intersection rings defined by quadrics, $\mathcal{R}(I)$ is Koszul and Cohen-Macaulay.
\item Since $\mathcal{R}(I)$ is Koszul, reg$_B(I^j)=rj$ for $j \geq 1$ follows from  Remark \ref{Koszulness_linear} (iii).
 
\end{enumerate}

\end{proof}

\begin{thm} \label{pf_case2.2}
Let $I=Pf_{n-1}(X_4)$. Then $\mathcal{R}(I)_{\Delta}$ is Koszul and Cohen-Macaulay for all $\Delta=(c,e)$, $c,e >0$.
\end{thm}

\begin{proof}
 Koszulness of $\mathcal{R}(I)_{\Delta}$ for all $\Delta$ follows from Proposition \ref{pf_case2} and Remark \ref{Koszulness_linear} (ii). 
 Since $\mathcal{R}(I)$ is a complete intersection domain, $\mathcal{R}(I)_{\Delta}$ is Cohen-Macaulay for all $\Delta$ from Proposition \ref{CM_diagonals}. 
\end{proof}

\vspace{2mm}

 \noindent {\bf Acknowledgement.} The first author is partially supported by the MATRICS grant, SERB India. The second author is financially supported by the INSPIRE fellowship, DST, India. The authors are also grateful to the anonymous referee for several suggestions which have helped improve the presentation of the article.

\end{document}